\documentclass{birkjour}

\usepackage{graphicx}
\usepackage[ansinew]{inputenc}    % accents i altres
\usepackage{verbatim}
\usepackage{amsmath,amstext,amssymb,}
\newtheorem{thm}{Theorem}[section]

\newtheorem{lemma}[thm]{Lemma}
\newtheorem{proposition}[thm]{Proposition}

\newtheorem{remark}[thm]{Remark}

\begin{document}

\title[Moutard hyperquadrics]{Moutard hyperquadrics and generalized Darboux directions}

\author[F.Py]{Fernanda Py Silva Cordeiro}

\address{Departamento de Matem\'atica-PUC-Rio, Rio de Janeiro, Brazil.}
\email{fernanda.py95@gmail.com}

\author[M. Craizer]{Marcos Craizer}

\address{Departamento de Matem\'atica-PUC-Rio, Rio de Janeiro, Brazil.}
\email{craizer@puc-rio.br}

\begin{abstract}
The higher order contact of a quadric with a surface in $3$-space at a non-degenerate point is obtained by the Moutard quadric in the Darboux direction. In this paper, we discuss the extension of this result to hypersurfaces in arbitrary dimensions. 
\end{abstract}


\thanks{This article is part of the Ph.D. thesis of the first author under the supervision of the second author. Both authors are thankful to CAPES and CNPq for financial support during the preparation of this paper.}


\date{April 25, 2024}

\maketitle

%%%%%%%%%%%%%%%%%%%%%%%%%%%%%%%%%%%%%%%%%%%%%%%%%%%%%%%%%%%%%%%%%%%%%%%%%%%%%%%%

%%%%%%%%%%%%%%%%%%%%%%%%%%%%%%%%%%%%%%%%%%%%%%%%%%%%%%%%%%%%%%%%%%%%%%%%%%%%%%%%
\section{Introduction}

The contact of a surface with some models gives insigths on the geometry of surfaces. For example, the contact with planes helps
understanding the geometry of the parabolic set and of the asymptotic directions, while the contact with spheres helps understanding the geometry at umbilical points (\cite{IFRT}). In projective geometry, it is natural to consider the contact with quadrics, or, in classical terminology, the intersection of a surface with osculating quadrics. 

Consider a non-degenerate and non-quadratic point of a surface in the $3$-dimensional space. Then there exists a $3$-parameter family of quadrics with a contact of order $2$ with the surface at the point. Among them, there are $1$-parameter families, one in the hyperbolic case and three in the elliptic case, for which the third order contact is a perfect cube. For these families, the contact of the quadric with the surface at the point is of type $E_6$, except for one value of the parameter, for which the contact is of type $E_7$. 
It turns out that these latter quadrics coincide with the classical Moutard quadrics in Darboux directions.  In this paper we shall study the generalization of Moutard quadrics in Darboux directions for hypersurfaces in arbitrary dimensions.

The Moutard quadric is a classical object of projective differential \break geometry (\cite{Buchin2},\cite{Cambraia},\cite{Green},\cite{Moutard}). 
Consider a  point $p$ of a non-degenerate surface in $M\subset\mathbb{R}^3$, and fix a direction ${\mathbf v}\in T_pM$. The one-parameter family of planes $\pi_{\lambda}$ containing ${\mathbf v}$ determine sections $M_{\lambda}$ of $M$, and we can consider
the conics $Q_{\lambda}\subset\pi_{\lambda}$ with a $4$-order contact with $M_{\lambda}$ at $p$. It turns out that the union of these conics, 
$Q=\cup_{\lambda}Q_{\lambda}$, is a quadric, called the {\it Moutard quadric} of $M$ at $p$ with direction ${\mathbf v}$. 
If we relax the condition of $4$-order contact and consider only $3$-order contact, we obtain a one-parameter family 
of quadrics $Q(\beta)$, $\beta\in\mathbb{R}$, that is called the {\it Moutard pencil of quadrics} at $p$ with direction ${\mathbf v}$. 

The contact function $\phi$ of any quadric of the Moutard pencil with the surface has vanishing $2$-jet. Moreover, there 
are three special directions ${\mathbf v}\in T_pM$, maybe two of them complex, such that the $3$-jet of $\phi$ is a perfect cube. These directions are called the {\it Darboux} directions of $p$. It is well-known that the Darboux directions are exactly the directions
of vanishing of the cubic form at $p$. 

For hypersurfaces, one can define in a similar way the Moutard hyperquadrics (\cite{Buchin}). For a point $p$
in a non-degenerate hypersurface $M\subset\mathbb{R}^{n+1}$ and ${\mathbf v}\in T_pM$, there is a $(n-1)$-parameter family
of $2$-planes $\pi_{\lambda}$, $\lambda\in\mathbb{R}^{n-1}$, containing ${\mathbf v}$. As above, we can consider
the conics $Q_{\lambda}\subset\pi_{\lambda}$ with a $4$-order contact with the section $M_{\lambda}=\pi_{\lambda}\cap M$ at $p$. The union of these conics, 
$Q=\cup_{\lambda}Q_{\lambda}$, is a hyperquadric, called the {\it Moutard hyperquadric} of $M$ at $p$ with direction ${\mathbf v}$. As in the two dimensional case, if we relax the condition of $4$-order contact and consider only $3$-order contact, we obtain a one-parameter family 
of quadrics $Q(\beta)$, $\beta\in\mathbb{R}$, that is called the {\it Moutard pencil of hyperquadrics} at $p$ with direction ${\mathbf v}$.

One can generalize the concept of Darboux directions to hypersurfaces as follows: 
For ${\mathbf v}\in T_pM$, consider a basis $\{{\mathbf v},{\mathbf v}_2,...,{\mathbf v}_n\}$ of $T_pM$ such that ${\mathbf v}_{\sigma}$ is orthogonal to ${\mathbf v}$ in the Blaschke metric at $p$, for any $\sigma\in\{2,...,n\}$.
Such a basis can be extended to a frame in a neighborhood of $p$, and we consider coordinates $(x_1,...,x_n)$ with respect to this frame. We shall say that ${\mathbf v}$ is a {\it generalized Darboux direction} if 
the $3$-jet of the contact of any Moutard hyperquadric in the Moutard pencil at $p$ in the direction ${\mathbf v}$ with the hypersurface depends only on the coordinates $(x_2,...,x_n)$.

In this paper we discuss some properties of the Moutard hyperquadrics. We prove for example that the sections of Moutard hyperquadrics coincide with the Moutard hyperquadrics of the sections of the hypersurface. Our main result is the relation between generalized Darboux directions and the cubic form $(C_{ijk})$. As a generalization of the case of surfaces, we show that
$v$ is a generalized Darboux direction if and only if $C_{1\sigma\sigma}=0$, for any $\sigma\in\{2,...,n\}$. 

The paper is organized as follows: In Section 2 we review the contact properties of Moutard quadrics. In Section 3 we discuss some properties of Moutard hyperquadrics, while in Section 4 we prove the relation between generalized Darboux directions and cubic forms. 

\section{Moutard quadrics and Darboux directions for surfaces}

In this section, we review the well-known results concerning the Moutard quadric of a non-degenerate surface $M\subset\mathbb{R}^3$ at a fixed point $p\in M$ and a fixed direction $\mathbf{v}\in T_pM$. These results can be found in, for example, \cite{Green}.
We shall assume, without loss of generality, that $p=(0,0,0)$, $\mathbf{v}=(1,0,0)$ and $T_pM$ is the plane $x_3=0$. We may also assume that the Blaschke metric $H$ at the origin is given by 
$H_{11}=1$, $H_{22}=\epsilon$, $H_{12}=0$,
where $\epsilon=\pm 1$. Under these hypothesis, the surface $M$ in a neighborhood of $p$ is given by $x_3=f(x_1,x_2)$, where
\begin{equation}\label{eq:Surface}
f(x_1,x_2)=\frac{1}{2}\left(x_1^2+\epsilon x_2^2\right)+\frac{1}{3}\sum K_{\sigma\tau\rho}x_{\sigma}x_{\tau}x_{\rho}+\frac{1}{12}\sum H_{\sigma\tau\rho\mu}x_{\sigma}x_{\tau}x_{\rho}x_{\mu}+O(5),
\end{equation}
with $1\leq \sigma,\tau,\rho,\mu \leq 2$. 

\subsection{The Moutard pencil of quadrics}

The {\it Moutard pencil of quadrics} associated to the $x_1$-axis is defined by $g=0$, where
\begin{equation}\label{eq:MoutardPencil2}
g(x_1,x_2,x_3)=-x_{3}+\frac{1}{2}\left(x_1^2+\epsilon x_2^2\right)+\frac{2}{3}K_{111}x_1x_{3}+ 2K_{112} x_{2}x_{3}+\beta x_{3}^2,\ \beta\in\mathbb{R},
\end{equation}
(\cite{Buchin2},\cite{Green}). We shall denote these quadric by $Q(\beta)$. Among these quadrics, we can distinguish the classical Moutard quadric, $Q$, 
given by 
\begin{equation}\label{eq:DefineBeta2}
\beta=\frac{1}{3}H_{1111}-\frac{8}{9}K_{111}^2.
\end{equation}

A plane $\pi=\pi_{\lambda}$ that contains the $x_1$-axis has equation $x_3=\lambda x_2$, $\lambda\in\mathbb{R}$. Denote by $\gamma=\gamma(\lambda)$
and $\eta=\eta(\lambda,\beta)$ the sections of $M$ and $Q(\beta)$, respectively, determined by $\pi$. Then, for any $\beta$ and any $\lambda$,
the contact of $\gamma$ and $\eta$ is of order $3$. Moreover, for $\beta$ given by \eqref{eq:DefineBeta2} and any $\lambda$, the contact is of order $4$. 

The contact function $\phi$ is defined as 
$$
\phi(x_1,x_2)=g(x_1,x_2,f(x_1,x_2)),
$$
where $f(x_1,x_2)$ is given by Equation \eqref{eq:Surface} and $g$ is given by Equation \eqref{eq:MoutardPencil2}.
One can verify that $\phi$ has zero $2$-jet 
and its $3$-jet $\phi_3$ is given by
\begin{equation}\label{eq:Phi32}
\phi_3(x_1,x_2)=b_{122}x_1x_2^2+b_{222}x_2^3.
\end{equation}
One can also verify that the coefficient $b_{1111}$ of $x_1^4$ in the Taylor expansion of $\phi$ at the origin is zero if and only if $\beta$ is given by \eqref{eq:DefineBeta2}.

\subsection{Darboux directions and cubic forms}

One can check that the coefficient $b_{122}$ of the contact function $\phi$ is given by 
\begin{equation}\label{eq:ContactCubic2}
b_{122}=\frac{\epsilon}{3}K_{111}-K_{122}.
\end{equation}
Note that, by Equation \eqref{eq:Phi32}, the $3$-jet $\phi_3$ of $\phi$ is a perfect cube if and only if $b_{122}=0$. We say that the $x_1$ axis is a Darboux direction of $M$ at the origin if $b_{122}=0$.

The cubic form of $M$ at the origin is given by 
\begin{equation}\label{eq:CubicForm2}
C_{111}=\frac{1}{2}K_{111}-\frac{3\epsilon}{2}K_{122}.
\end{equation}
From Equations \eqref{eq:ContactCubic2} and \eqref{eq:CubicForm2} we have that 
$$
C_{111}=\frac{3\epsilon}{2}b_{122}.
$$
We conclude that $x_1$ is a Darboux direction if and only if $C_{111}=0$. More generally, it is well-known that a direction ${\mathbf v}\in T_pM$ is a Darboux direction if and only if $C({\mathbf v},{\mathbf v},{\mathbf v})=0$.

By apolarity, the condition $C_{111}=0$ is equivalent to $C_{122}=0$. Thus the axis $x_1$ is a Darboux direction if and only if $C_{122}=0$. More generally, a direction ${\mathbf v}\in T_pM$ is a Darboux direction if and only if $C({\mathbf v},{\mathbf v}^{\perp},{\mathbf v}^{\perp})=0$, where ${\mathbf v}^{\perp}$ is orthogonal to ${\mathbf v}$ in the Blaschke metric.

\subsection{Higher order contact of quadrics with a surface}

In this section we assume that $\mathbf{v}=(1,0,0)$ is a Darboux direction. By previous section, this implies that $K_{111}=K_{122}=0$. Thus 
\begin{equation}\label{eq:Surface2}
f(x_1,x_2)=\tfrac{1}{2}\left(x_1^2+\epsilon x_2^2\right)+K_{112} x_1^2x_2+\tfrac{K_{222}}{3}x_2^3 +\tfrac{1}{12}\sum H_{\sigma\tau\rho\mu}x_{\sigma}x_{\tau}x_{\rho}x_{\mu}+O(5),
\end{equation}
It is easy to see that the $3$-parameter family of quadrics with contact of order $2$ with $M$ at the origin is given by
\begin{equation*}
x_3=\tfrac{1}{2}\left(x_1^2+\epsilon x_2^2\right)+Ex_1x_3+Fx_2x_3+\beta x_3^2.
\end{equation*}
The contact of order $3$ is given by
$$
\phi_3(x_1,x_2)=K_{112} x_1^2x_2+\tfrac{K_{222}}{3}x_2^3 -\tfrac{E}{2}x_1(x_1^2+\epsilon x_2^2)-\tfrac{F}{2}x_2(x_1^2+\epsilon x_2^2).
$$
If $E=0$, $F=2K_{112}$, then 
$$
\phi_3(x_1,x_2)=\tfrac{1}{3}(K_{222}-\epsilon K_{112})x_2^3
$$
is a perfect cube and the osculating quadric becomes
\begin{equation}\label{eq:QuadricsContactOrder2}
x_3=\tfrac{1}{2}\left(x_1^2+\epsilon x_2^2\right)+2K_{112}x_2x_3+\beta x_3^2,
\end{equation}
thus belonging to the Moutard pencil. 
Summarizing, every quadric in the Moutard pencil associated with the Darboux direction have a perfect cube as its third order contact with the surface. 

Let us consider the contact $\phi_4$ of order $4$ of the surface \eqref{eq:Surface2} with the quadric \eqref{eq:QuadricsContactOrder2}. Denote $c=\tfrac{1}{3}(K_{222}-\epsilon K_{112})$. We have that
$$
\phi_4(x_1,x_2)=cx_2^3+\tfrac{1}{12}\sum H_{\sigma\tau\rho\mu}x_{\sigma}x_{\tau}x_{\rho}x_{\mu}-2K_{112}^2x_1^2x_2^2-\tfrac{2}{3}K_{112}K_{222}x_2^4
$$
$$
-\tfrac{\beta}{4}\left( x_1^2+\epsilon x_2^2  \right)^2.
$$
Now consider a change of variables of the form $x_1=u_1$ and 
$$
x_2=u_2+\alpha u_1^2+\beta u_1u_2+\gamma u_2^2,
$$
By an appropriate choice of $\alpha$, $\beta$ and $\gamma$, the contact function can be expressed as 
$$
\phi_4(u_1,u_2)=cu_ 2^3+\tfrac{1}{3}H_{1112}u_2u_1^3+\tfrac{1}{12}\left(H_{1111}-3\beta \right)u_1^4
$$
If the coefficient of $u_2^4$ is not zero, the contact function is of type $E_6$. In other words, if the quadric is in the Moutard pencil but is not the Moutard quadric, then the contact is of type $E_6$. 
If the coefficient of $u_2^4$ is zero, $\beta=\tfrac{H_{1111}}{3}$, we are with the Moutard quadric. In this case, if we assume the generic condition $H_{1112}\neq 0$, the contact is of type $E_7$.

\section{Moutard Hyperquadrics}

Let $x=(x_1,...x_n)$ and consider a non-degenerate hypersurface defined by $x_{n+1}=f(x)$, where
\begin{equation}\label{eq:Hypersurface}
f(x)=\frac{1}{2}\sum H_{\sigma\tau}x_{\sigma}x_{\tau}+\frac{1}{3}\sum K_{\sigma\tau\rho}x_{\sigma}x_{\tau}x_{\rho}+\frac{1}{12}\sum H_{\sigma\tau\rho\mu}x_{\sigma}x_{\tau}x_{\rho}x_{\mu}+O(5).
\end{equation}
Since the hypersurface is non-degenerate, we may assume, and we shall do it, that, at $x=0$, $H_{\sigma\sigma}=\epsilon_{\sigma}$, where $\epsilon_{\sigma}=\pm 1$, and $H_{\sigma\tau}=0$, if $\sigma\neq\tau$.

\subsection{Moutard pencil of hyperquadrics}

Consider $2$-planes containing the  $x_{1}$-axis. Each such plane can be written as
\begin{equation}\label{eq:2nplane}
x_{\sigma}=\lambda_{\sigma} x_{n+1},\ \ \sigma=2,...n.
\end{equation}

\begin{lemma}\label{lemma:a2a3a4}
The projection of the section of the hypersurface \eqref{eq:Hypersurface} by the $2$-plane \eqref{eq:2nplane} in the plane $(x_1,x_{n+1})$ is given by
\begin{equation*}
x_{n+1}=a_2 x_1^2+a_3 x_1^3+a_4 x_1^4+O(5),
\end{equation*}
where 
\begin{equation*}\label{a2a3}
a_2=\frac{1}{2}\epsilon_{1},\ a_3=\frac{1}{3}K_{111},
\end{equation*}
and 
\begin{equation*}\label{a4}
a_4=\frac{1}{12}H_{1111}+\frac{\epsilon_{1}}{2}\sum_{\sigma=2}^n K_{11\sigma}\lambda_{\sigma}+
\frac{1}{8}\sum_{\sigma=2}^n \epsilon_{\sigma}\lambda_{\sigma}^2.
\end{equation*}
\end{lemma}

\begin{lemma}\label{lemma:OsculatingConic}
Consider the planar curve
$$
x_{n+1}=a_2 x_1^2+a_3 x_1^3+a_4 x_1^4+O(5).
$$
Then the osculating conic at the origin is given by
\begin{equation*}\label{eq:PlanarOsculatingConic}
x_{n+1}=a_2x_1^2+\frac{a_3}{a_2}x_{n+1}x_1+\left( \frac{a_2a_4-a_3^2}{a_2^3}  \right)x_{n+1}^2.
\end{equation*}
\end{lemma}

\smallskip\noindent
The proofs of Lemmas \ref{lemma:a2a3a4} and \ref{lemma:OsculatingConic} are easy and we left them to the reader.
From these lemmas, we obtain the equation of the osculating conic of the section of the hypersurface \eqref{eq:Hypersurface} by the $2$-plane \eqref{eq:2nplane}. We then eliminate $\lambda_{\sigma}$ in this equation using Equations \eqref{eq:2nplane} to obtain the equation $g(x,x_{n+1})=0$, where
\begin{equation}\label{eq:MoutardPencil}
g=-x_{n+1}+\tfrac{1}{2}\sum_{\sigma=1}^n \epsilon_{\sigma} x_{\sigma}^2+\tfrac{2}{3}K_{111}\epsilon_1x_1x_{n+1}+
2\epsilon_{1}\sum_{\sigma=2}^n K_{11\sigma}x_{\sigma} x_{n+1}+
\beta x_{n+1}^2,
\end{equation}
with
\begin{equation}\label{eq:BetaMoutard}
\beta=\tfrac{1}{9}\left( 3H_{1111}-8\epsilon_1K_{111}^2\right).
\end{equation}

\smallskip\noindent
The one-parameter family of hyperquadrics defined by Equation \eqref{eq:MoutardPencil} is called the {\it Moutard pencil} of hyperquadrics (\cite{Buchin}).
Among the hyperquadrics of the Moutard pencil, we distinguish the {\it Moutard hyperquadric}, obtained with $\beta$
given by Equation \eqref{eq:BetaMoutard}.

\subsection{Contact properties}

The contact function $\phi$ is defined as 
$$
\phi(x)=g(x,f(x)),
$$
where $f(x)$ is given by Equation \eqref{eq:Hypersurface}. It is not difficult to see that the $2$-jet of $\phi$ at the origin is zero.

Denote $\bar{x}=(x_2,...,x_n)$. The $3$-jet $\phi_3$ of $\phi$ can be written as 
$$
\phi_3(x_1,\bar{x})=x_1q(x)+p_3(\bar{x}),
$$
where $q(x)$ is quadratic and $p_3(\bar{x})$ cubic. We write
$$
q(x)= b_{111}x_1^2+\sum_{\sigma=2}^n b_{11\sigma}x_1x_{\sigma} +\sum_{\sigma=2}^nb_{1\sigma\sigma}x_{\sigma}^2 +\sum_{\sigma\neq\tau} b_{1\sigma\tau}x_{\sigma}x_{\tau}
$$
It is not difficult to verify that $b_{111}=0$ and $b_{11\sigma}=0$, for any $2\leq\sigma\leq n$.
Moreover, 
\begin{equation}\label{eq:ContactCubic}
b_{1\sigma\sigma}=\frac{\epsilon_1\epsilon_{\sigma}}{3}K_{111}-K_{1\sigma\sigma},\ \ b_{1\sigma\tau}=-2K_{1\sigma\tau}.
\end{equation}
We have thus showed that 
\begin{equation}\label{eq:quadratic}
q(x)=\sum_{\sigma=2}^nb_{1\sigma\sigma}x_{\sigma}^2 +\sum_{\sigma\neq\tau} b_{1\sigma\tau}x_{\sigma}x_{\tau}, 
\end{equation}
with $b_{1\sigma\sigma}$ and $b_{1\sigma\tau}$ given by Equations \eqref{eq:ContactCubic}. We say that the $x_1$-axis is a {\it generalized Darboux direction} if $q(x)=0$, or equivalently, $b_{1\sigma\sigma}=b_{1\sigma\tau}=0$.

The coefficient $b_{1111}$ of $x_1^4$ in the Taylor expansion of $\phi$ is given by
$$
b_{1111}=-\tfrac{1}{12}H_{1111}+\tfrac{2}{9}\epsilon_1K_{111}^2+\tfrac{1}{4}\beta=0.
$$
We conclude that $b_{1111}=0$ if and only if $\beta$ is given by Equation \eqref{eq:BetaMoutard}, which means that
we have chosen the Moutard hyperquadric among the hyperquadrics of the Moutard pencil.

\subsection{Sectional property}

For a fixed $\bar\lambda=(\lambda_{k},...,\lambda_n)$, denote by $\pi_{\bar\lambda}$ the $k$-space defined by
\begin{equation}\label{eq:kspace}
x_{\sigma}=\lambda_{\sigma}x_{n+1}, \ \ \sigma=k,...n, 
\end{equation}
containing the $x_1$-axis and denote by $M_{\bar\lambda}$ the intersection of $\pi_{\bar\lambda}$ with the hypersurface $M$. A $2$-plane in the $k$-space defined by \eqref{eq:kspace} containing the $x_1$-axis is defined by the parameters 
$\lambda_0=(\lambda_2,...,\lambda_{k-1})$, through the equations
\begin{equation*}\label{eq:2planek}
x_{\sigma}=\lambda_{\sigma}x_{n+1}, \ \ \sigma=2,...k-1.
\end{equation*}
The same $2$-plane in $\mathbb{R}^{n+1}$ is defined by $\lambda=\lambda_0\cup\bar\lambda$  through Equations
\eqref{eq:2nplane}.
%\begin{equation}\label{eq:2planen}
%x_{\sigma}=\lambda_{\sigma}x_{n+1}, \ \ \sigma=2,...n.
%\end{equation}

\begin{proposition}
The Moutard quadric of $M_{\bar\lambda}$ in the direction $x_1$ coincides with the intersection of the Moutard hyperquadric of $M$ in the direction $x_1$ with the $k$-space $\pi_{\bar\lambda}$. 
\end{proposition}

\begin{proof}
The intersection of a $2$-plane $\pi_{\lambda_0}$ contained in the $k$-space defined by Equations \eqref{eq:kspace} and containing the $x_1$-axis with the Moutard quadric of $M_{\bar\lambda}$ has a $4$-order contact with the corresponding section of $M_{\bar\lambda}$. Similarly, the intersection of a section defined by $\lambda$ containing the $x_1$-axis with the Moutard hyperquadric of $M$ has a $4$-order contact with the corresponding section of $M$. Thus both conics coincide and the proposition is proved. 
\end{proof}

\section{Generalized Darboux Directions}

Consider a non-degenerate hypersurface defined by $x_{n+1}=f(x)$, where $f(x)$ is given by Equation \eqref{eq:Hypersurface}.

\subsection{Cubic forms}

\begin{proposition}\label{prop:CubicForm}
The cubic form of the hypersurface \eqref{eq:Hypersurface} at the origin is given by $C_{\sigma\tau\rho}x_{\sigma}x_{\tau}x_{\rho}$, where
\begin{equation*}\label{Cubic1}
C_{\sigma\sigma\sigma}=\frac{2(n-1)}{n+2}K_{\sigma\sigma\sigma}-\frac{6}{n+2}\left( \sum_{\tau}\epsilon_{\sigma}\epsilon_{\tau}K_{\sigma\tau\tau} \right).
\end{equation*}
\begin{equation*}\label{Cubic2}
C_{\sigma\tau\tau}=\frac{2(n+1)}{n+2}K_{\sigma\tau\tau}-\frac{2\epsilon_{\sigma}\epsilon_{\tau}}{n+2}K_{\sigma\sigma\sigma}
-\frac{2\epsilon_{\tau}}{n+2}\sum_{\rho}\epsilon_{\rho}K_{\sigma\rho\rho}.
\end{equation*}
\begin{equation*}\label{Cubic3}
C_{\sigma\tau\rho}=2K_{\sigma\tau\rho}, 
\end{equation*}
where $\sigma\neq\tau\neq\rho$.
Observe that the apolarity condition 
\begin{equation*}\label{apolarity}
\epsilon_{\sigma}C_{\sigma\sigma\sigma}+\sum_{\tau}\epsilon_{\tau}C_{\sigma\tau\tau}=0,
\end{equation*}
holds for any $\sigma$.
\end{proposition}

For a proof of this proposition, see Appendix.

\subsection{Generalized Darboux directions}

By the apolarity condition, the $n$ conditions $C_{1\sigma\sigma}=0$, $\sigma=1,...,n$, are redundant, and we need to consider only $(n-1)$ among them. Next lemma is elementary:   

\begin{lemma}\label{lemma:bc1}
The linear system of $(n-1)$ equations $b_{1\sigma\sigma}=0$ in the $(n-1)$ variables $K_{1\sigma\sigma}$,
is equivalent to the system $C_{1\sigma\sigma}=0$, $\sigma=2,...,n$. 
\end{lemma}

\begin{proof}
Assume first that $b_{1\sigma\sigma}=0$, $\sigma=2,...,n$. Then
\begin{equation}\label{eq:1}
\frac{1}{3}\epsilon_1\epsilon_{\sigma}K_{111}-K_{1\sigma\sigma}=0.
\end{equation}
Fix $\tau\in\{2,..,n\}$. For $\sigma=\tau$, multiply Equation \eqref{eq:1} by $-(n+1)$. For $\sigma=\rho$, $\rho\in\{2,...,n\}$, $\rho\neq\tau$, multiply 
Equation \eqref{eq:1} by $\epsilon_{\tau}\epsilon_{\rho}$ and then sum all to obtain
$$
(n+1)K_{1\tau\tau}-\epsilon_1\epsilon_{\tau}K_{111}-\epsilon_{\tau}\sum_{\rho\neq\tau}\epsilon_{\rho}K_{1\rho\rho}=0,
$$
proving that $C_{1\tau\tau}=0$. Conversely, assume that $C_{1\sigma\sigma}=0$, for $1\leq\sigma\leq n$. The equation $C_{111}=0$ can be written as
\begin{equation}\label{eq:C1110}
K_{111}=\sum_{\tau=2}^n\epsilon_1\epsilon_{\tau}K_{1\tau\tau}=0.
\end{equation}
On the other hand, the Equation  $C_{1\sigma\sigma}=0$ can be written as
\begin{equation}\label{eq:C1SigmaSigma0}
\epsilon_1\epsilon_{\sigma}K_{111}=5K_{1\sigma\sigma}-\sum_{\tau\neq\sigma}^n\epsilon_{\sigma}\epsilon_{\tau}K_{1\tau\tau}=0.
\end{equation}
Multiplying Equation \eqref{eq:C1110} by $\epsilon_{\sigma}$, Equation \eqref{eq:C1SigmaSigma0} by $\epsilon_{1}$ and summing we obtain
$$
2\epsilon_{\sigma}K_{111}=6\epsilon_1K_{1\sigma\sigma},
$$
thus proving that $b_{1\sigma\sigma}=0$. 
\end{proof}

\begin{lemma}\label{lemma:bc2}
The linear system of $\tfrac{n(n-1)}{2}$ equations $b_{1\sigma\tau}=0$ in the $\tfrac{n(n-1)}{2}$ variables $K_{1\sigma\tau}$,
is equivalent to the system $C_{1\sigma\tau}=0$, $\sigma,\tau=2,...,n$. 
\end{lemma}

\begin{proof}
If $\sigma\neq\tau$, $C_{1\sigma\tau}=0$ and $b_{1\sigma\tau}=0$ are equivalent to $K_{1\sigma\tau}=0$. The result now follows from Lemma \ref{lemma:bc1}. 
\end{proof}

Lemma \ref{lemma:bc2} implies the main result of the paper:

\begin{proposition}
The $x_1$ axis is a generalized Darboux direction if and only if $C_{1\sigma\tau}=0$, for $2\leq \sigma,\tau\leq n$.
\end{proposition}

\begin{proof}
Observe that, by Equation \eqref{eq:quadratic}, the $3$-jet of the contact function $\phi$ is independent of $x_1$ if and only if $b_{1\sigma\tau}=0$, for $2\leq \sigma,\tau\leq n$. Now the result follows directly from Lemma \ref{lemma:bc2}.
\end{proof}

\begin{remark} 
The above proposition shows an important difference between surfaces and higher dimensional hypersurfaces: Generically, 
the generalized Darboux directions exist only at a submanifold of the hypersurface $M^n$. For example, if $n=3$, the generalized Darboux directions exist only at a surface $M^2\subset M^3$.  
%at a point
%must satisfy the three equations $C_{122}=C_{133}=C_{123}=0$, which can be impossible since each of these equations is homogeneous in three variables. 
\end{remark}

%\section*{\textbf{Acknowledgement}}

%This article is part of the Ph.D. thesis of the first author under the supervision of the second author. Both authors are thankful to CAPES and CNPq for financial support during the preparation of this paper.

%\subsection*{Funding:} Both authors had the support of CNPq (Conselho Nacional de Desenvolvimento Cient\'ifico e Teconol\'ogico - Brazil) and CAPES (Coordena\c c\~ao de Aperfei\c coamento de Pessoal de N\'ivel Superior - Brazil).

%\subsection*{Author Contribution:} Both authors wrote and review the article.

%\subsection*{Conflict of Interest:} The authors have no conflict of interest.

%\subsection*{Data Availability Statement:} This article is theoretical and do not use experimental data. 

\section{Appendix: The cubic form}\label{appendix}

In this appendix we prove Proposition \ref{prop:CubicForm}: Denote $\xi_0=(0,1)$ and $\epsilon=\Pi_{\sigma=1}^n\epsilon_{\sigma}$. Following \cite[p.45]{Nomizu}, the affine Blaschke normal of 
$$
\psi(x)=(x,f(x))
$$
is given by 
$$
\xi=\phi\xi_0+Z_1\psi_{x_1}+Z_2\psi_{x_2}+...+Z_n\psi_{x_n},
$$
where $\phi^{n+2}=|\det(D^2f)|=\epsilon\det(D^2f)$ and $Z_i$, $1\leq i\leq n$, is given by
\[
\left[D^2f\right]\cdot
\left[
\begin{array}{c}
Z_1\\ Z_2\\ ...\\ Z_n
\end{array}
\right]=
-\left[
\begin{array}{c}
\phi_{x_1}\\ \phi_{x_2}\\ ...\\ \phi_{x_n}
\end{array}
\right].
\]
Thus, at the origin, 
\[
\left[
\begin{array}{c}
\epsilon_1Z_1\\ \epsilon_2Z_2\\ ...\\ \epsilon_nZ_n
\end{array}
\right]=
-\frac{2}{n+2}\left[
\begin{array}{c}
\epsilon_1K_{111}+\epsilon_2K_{122}+...+\epsilon_nK_{1nn}\\ 
\epsilon_1K_{112}+\epsilon_2K_{222}+...+\epsilon_nK_{2nn}\\ ...\\ 
\epsilon_1K_{11n}+\epsilon_2K_{22n}+...+\epsilon_nK_{nnn}
\end{array}
\right].
\]
We conclude that
$$
Z_{\sigma}=-\frac{2}{n+2}\left( K_{\sigma\sigma\sigma}+\sum_{\tau\neq\sigma}\epsilon_{\sigma}\epsilon_{\tau}K_{\sigma\tau\tau}  \right).
$$
Moreover, the affine Blaschke metric is given by 
$$
h_{\sigma\tau}=\frac{1}{\phi}\cdot\frac{\partial^2f}{\partial x_{\sigma}\partial x_{\tau}}.
$$
At the origin $h_{\sigma\sigma}=\epsilon_{\sigma}$. The induced connection $\nabla$ is given by
$$
\nabla_{X_{\sigma}}X_{\tau}= -\frac{1}{\phi}\cdot\frac{\partial^2f}{\partial x_{\sigma}\partial x_{\tau}}\left(Z_1X_1+Z_2X_2+...+Z_nX_n \right),
$$
where $X_{\sigma}=\frac{\partial}{\partial x_{\sigma}}$. Thus, at the origin, $\nabla_{X_{\sigma}}X_{\tau}=0$, if $\sigma\neq\tau$. Moreover,
$$
\nabla_{X_{\sigma}}X_{\sigma}=-\epsilon_{\sigma}\left(Z_1X_1+Z_2X_2+...+Z_nX_n \right).
$$

\bigskip\noindent
We can now calculate the cubic form:
$$
C_{\sigma\sigma\sigma}=\frac{\partial h_{\sigma\sigma}}{\partial x_{\sigma}}-2h\left(\nabla_{X_{\sigma}}X_{\sigma},X_{\sigma}  \right)=-\epsilon_{\sigma}\phi_{x_{\sigma}}+2K_{\sigma\sigma\sigma}+2Z_{\sigma}=2K_{\sigma\sigma\sigma}+3Z_{\sigma}
$$
$$
=-\frac{6}{n+2}\left( K_{\sigma\sigma\sigma}+\sum_{\tau}\epsilon_{\sigma}\epsilon_{\tau}K_{\sigma\tau\tau}  \right)+2K_{\sigma\sigma\sigma}=\frac{2(n-1)}{n+2}K_{\sigma\sigma\sigma}-\frac{6}{n+2}\left( \sum_{\tau}\epsilon_{\sigma}\epsilon_{\tau}K_{\sigma\tau\tau} \right).
$$

\bigskip\noindent
Moreover, at the origin
$$
C_{\sigma\sigma\tau}=\frac{\partial h_{\sigma\sigma}}{\partial x_{\tau}}-2h\left(\nabla_{X_{\tau}}X_{\sigma},X_{\sigma}  \right)=-\epsilon_{\sigma}\phi_{x_{\tau}}+2K_{\sigma\sigma\tau}=\epsilon_{\sigma}\epsilon_{\tau}Z_{\tau}+2K_{\sigma\sigma\tau}
$$
$$
=-\frac{2\epsilon_{\sigma}\epsilon_{\tau}}{n+2}\left( K_{\tau\tau\tau}+\sum_{\rho}\epsilon_{\rho}\epsilon_{\tau}K_{\rho\rho\tau}\right)+2K_{\sigma\sigma\tau}
$$
$$
=\frac{2(n+1)}{n+2}K_{\sigma\sigma\tau}-\frac{2\epsilon_{\sigma}\epsilon_{\tau}}{n+2}K_{\tau\tau\tau}
-\frac{2\epsilon_{\sigma}}{n+2}\sum_{\rho}\epsilon_{\rho}K_{\rho\rho\tau}.
$$

\bigskip\noindent
Finally, at the origin,
$$
C_{\sigma\tau\rho}=\frac{\partial h_{\sigma\tau}}{\partial x_{\rho}}-h\left(\nabla_{X_{\rho}}X_{\sigma},X_{\tau}  \right)-h\left(\nabla_{X_{\rho}}X_{\tau},X_{\sigma}  \right)=2K_{\sigma\tau\rho},
$$
thus completing the proof of the proposition.

\end{document}